\documentclass[12pt]{article}

\usepackage[utf8]{inputenc}
\usepackage{amsmath, amsthm, amssymb, color, mathbbol}
\usepackage[shortlabels]{enumitem}
\usepackage{graphicx}
\usepackage{makecell}
\usepackage[ruled]{algorithm2e}
\usepackage{multicol}

\usepackage[table,xcdraw]{xcolor}
\usepackage{tikz}
\usepackage{array,amscd}
\usepackage{amsfonts}
\usepackage[T1]{fontenc} 
\usepackage{enumerate}
\usepackage{appendix}
\usepackage[arrow, matrix, curve]{xy}
\usepackage[square,comma,sort,numbers]{natbib}
\usepackage{booktabs}
\usepackage{siunitx}
\usepackage{multirow}
\usepackage[font=small,labelfont=bf]{caption}

\usepackage{tikz}
\usetikzlibrary{snakes}
\usetikzlibrary {shapes}
\usetikzlibrary {arrows}
\usetikzlibrary {positioning}
\usetikzlibrary {calc}

\oddsidemargin \evensidemargin
\usepackage{enumitem}
\setenumerate{leftmargin=*}
\allowdisplaybreaks[3]	

\newtheorem{theorem}{Theorem}

\newtheorem{lemma}[theorem]{Lemma}

\newtheorem{definition}[theorem]{Definition}

\newtheorem{example}[theorem]{Example}

\definecolor{darkgreen}{rgb}{0,0.4,0}
\definecolor{MyBlue}{rgb}{0,0.08,0.7} 
\definecolor{MyRed}{rgb}{0.85,0.08,0} 

\newcommand{\bi}{{\, \color{MyRed}\leftrightarrow}\,}
\newcommand{\toblue}{{\, \color{MyBlue}\to}\,}

\renewcommand{\bi}{\leftrightarrow}
\renewcommand{\toblue}{\to}

\newcommand\trans{T}
\DeclareMathOperator{\var}{Var}
\DeclareMathOperator{\tr}{tr}

\makeatletter
\renewcommand*\env@matrix[1][\arraystretch]{%
  \edef\arraystretch{#1}%
  \hskip -\arraycolsep
  \let\@ifnextchar\new@ifnextchar
  \array{*\c@MaxMatrixCols c}}
\makeatother

\title{\textbf{Structure Learning for Cyclic Linear \\ Causal Models}}

\author{Carlos Am\'{e}ndola, Philipp Dettling, Mathias Drton, Federica Onori, Jun Wu}

\date{}

\begin{document}

\maketitle

\begin{abstract}
We consider the problem of structure learning for linear causal models based on observational data.  We treat models given by possibly cyclic mixed graphs, which allow for feedback loops and effects of latent confounders.  Generalizing related work on bow-free acyclic graphs, we assume that the underlying graph is simple.  This entails that any two observed variables can be related through at most one direct causal effect and that (confounding-induced) correlation between error terms in structural equations occurs only in absence of direct causal effects.
We show that, despite new subtleties in the cyclic case, the considered simple cyclic models are of expected dimension and that a previously considered criterion for distributional equivalence of bow-free acyclic graphs has an analogue in the cyclic case.  Our result on model dimension justifies in particular score-based methods for structure learning of linear Gaussian mixed graph models, which we implement via greedy search.
\end{abstract}

\section{Introduction}

Inferring the structure of a causal model with feedback loops from observational data is a notoriously difficult---if not impossible---problem, particularly if one also seeks to guard against presence of latent confounders \cite{evans2018model,spirtes:2000}.  We consider this problem for linear causal models given by mixed graphs (or path diagrams) with directed and bidirected edges. As detailed in Section \ref{sec:background}, the vertices of such a graph correspond to the observed variables, and the directed edges encode structural equations that 
relate these variables up to stochastic noise.  The bidirected edges indicate possible correlations among the noise terms, as may be induced by latent confounders.

Much work has gone into algorithms that exploit conditional independence relations for learning the structure of causal models, or rather suitable equivalence classes of graphs encoding this structure; see, e.g., \cite{Rantanen20,hytt_UAI_14,Hyttinen3,richardson1996,
Forre}  
or also the review of Spirtes and Zhang in \cite[\S18]{handbook}.  While methods have been developed that use information about conditional independence relations also in settings  with feedback loops or latent variables, there is an inherent limitation to this approach as causal models with feedback loops or latent variables can generally not be characterized using conditional independence constraints alone \cite{VermaPearl90,EvansEtAl14,VanOmmenMooij_UAI_17,drton:robeva:202x}.  Alternatively, structure learning can be approached using score-based search techniques; see, e.g., \cite{MR3099113,MR1991085,solus2017consistency}.  For models with feedback loops and/or latent variables, however, the definition of an appropriate statistical score is non-trivial as the model parameters need not be identifiable and, consequently, the model dimension may differ from the number of parameters that are used to specify the model.  Although methods exist to detect identifiability or lack thereof \cite{foygel:draisma:drton:2012,div,kumor:2019}, it is generally unclear when a linear causal model with feedback loops or latent variables is of the dimension expected from a parameter count \cite{drton:algebraic}. 

Motivated by these difficulties, prior work has also made attempts to determine more tractable settings.  For instance, it has been shown that if an acyclic mixed graph is bow-free, i.e., no pair of nodes is joined by both a directed and a bidirected edge, then the parameters of the induced model are generically identifiable \cite{brito:pearl:2002}, which entails expected dimension.  Score-based methods for structure learning in this setting were proposed in \cite{nowzohour:2017}.

 In this paper we generalize results that have been obtained for bow-free acyclic mixed graphs to graphs that may have a cyclic directed part.   In other words, we consider mixed graphs that are simple, i.e., have at most one edge between any pair of nodes, but which need not be acyclic.   The presence of cycles brings about many new challenges \cite{bongers2016theoretical} and, in particular, generic identifiability cannot be guaranteed.  However, as our main results show, the considered simple cyclic models are of expected dimension (Theorem \ref{thm:dim}) and a previously considered criterion for distributional equivalence of bow-free acyclic graphs has an analogue in the cyclic case (Theorem \ref{thm:equivalence}).  Using the result on dimension, we propose a model selection score and associated greedy search techniques for structure learning of linear Gaussian mixed graphs. Numerical experiments indicate the need to carefully account for the large size of the set of models/graphs the search considers (Section \ref{sec:modelsel}).

  \section{Background} 
 \label{sec:background}
 
 Let $X=(X_i:i\in V)$ be a random vector whose coordinates correspond to observed variables.  The considered models assume $X$ to be the solution to
 a linear equation system of the form
\begin{equation}
  \label{eq:sem:general}
  X \;=\; \Lambda^\trans X + \varepsilon,
\end{equation}
where $\Lambda=(\lambda_{ij})\in\mathbb{R}^{V\times V}$ is a matrix of
unknown parameters, and $\varepsilon=(\varepsilon_i:i\in V)$ is a random vector whose coordinates represent stochastic noise.  Suppose $\varepsilon$ has (unknown) covariance matrix
$\Omega=(\omega_{ij})$. Assuming that $I-\Lambda$ is invertible ($I$ denotes the identity matrix), the system in~(\ref{eq:sem:general}) is solved uniquely by
$X = (I-\Lambda)^{-\trans}\varepsilon$ with covariance matrix
\begin{equation}
  \var[X] \,=\, (I - \Lambda)^{-\trans} \Omega (I -
  \Lambda)^{-1} \;=:\; \phi(\Lambda,
  \Omega).\label{eq:VarX}
\end{equation}
Specific models of interest place restrictions on the support of $\Lambda$ and $\Omega$, and this is naturally represented by a graph. More precisely, we adopt mixed graphs since we have two parameter matrices, $\Lambda$ and $\Omega$, whose rows and columns are
indexed by the same set $V$. 

A mixed graph with vertex set $V$ is a triple $G=(V,D,B)$ where
$D,B\subseteq V\times V$ are two sets of edges.  The set $D$ comprises  
ordered pairs $(i,j)$, $i\not=j$, that encode directed edges, which we also denote by $i\toblue j$.  Node $j$ is the head of such an edge.  The elements of $B$ are unordered pairs $\{i,j\}$ with $i\not=j$ that encode bidirected edges, also denoted by $i\bi j$.  These edges have no orientation, and
$i\bi j\in B$ if and only if $j\bi i\in B$.  It is convenient to
call both endpoints $i$ and $j$ heads of $i\bi j$.  A collider triple in $G=(V,D,B)$ is a triple of vertices $(i,j,k)$ such that there are edges between $i$ and $j$ and between $j$ and $k$, with $j$ being a head on both these edges.   In other words, the two edges form a path of the form $i \to j  \leftarrow k$, $i \bi j  \leftarrow k$, $i \to j  \bi k$ or $i \bi j  \bi k$.  We emphasize that whether $(i,j,k)$ is a collider triple does not depend on absence or presence of an edge between $i$ and $k$.   Finally, the skeleton of $G$ is the undirected graph obtained by replacing all edges in $(V, D \cup B)$ by undirected edges. 

In Section \ref{sec:sufficient} we will use the concept of treks. A trek is a path without collider triples and thus takes the form:
\begin{align*}
    v_l^{L} &\leftarrow \dots \leftarrow v_{1}^{L} \leftrightarrow v_{1}^{R} \rightarrow \dots \rightarrow v_{r}^{R},  \text{ or }\\
    &\,\,v_l^{L} \leftarrow \dots \leftarrow v_0  \rightarrow \dots \rightarrow v_{r}^{R}.
\end{align*}
Each trek can be decomposed in two directed paths, referred to as the right and the left side of the trek.  A trek is simple if its left- and right-hand side do not intersect, with the exception of an intersection at the top node $v_0$ for the second type of trek.

In the sequel, the key assumption we make about the considered graphs is that they be \emph{simple} mixed graphs.  This means that we do not allow more than one edge (of any type) between any two nodes.  Note that this allows for presence of directed cycles of length at least 3.
\smallskip

\begin{example}
  Two simple mixed graphs are displayed in Figure \ref{fig:two_graph}.  The first is a directed 3-cycle, which has no collider triples. The second graph has the collider triple $(1,3,2)$.  Both graphs have the same skeleton.
\end{example}

Let $\mathbb{R}^D$ be the set of real $V\times V$-matrices
$\Lambda=(\lambda_{ij})$ with support in $D$, that is, 
\begin{equation}
  \label{eq:R^D}
      \mathbb{R}^D = \big\{\,\Lambda \in \mathbb{R}^{V
      \times V} : \lambda_{ij} = 0 \ \text{ if } \ i \toblue j \notin D
    \,\big\}.
\end{equation}
Define $\mathbb{R}^D_\mathrm{reg}$ to be the subset of matrices
$\Lambda\in\mathbb{R}^D$ for which $I-\Lambda$ is invertible. Similarly, let
$\mathit{PD}$ be the cone of positive definite symmetric
$V\times V$-matrices, and define
$\mathit{PD}(B)$ to be the subcone with support over $B$, that
is, 
\begin{equation} \label{eq:PDB}
  \mathit{PD}(B)= \big\{\, \Omega=(\omega_{ij}) \in \mathit{PD} : \omega_{ij} = 0
  \ \text{ if } \
  i \neq j \mbox{ and } i \bi j \notin B\,\big\}.
\end{equation}

 \begin{figure}[t] 
  \center
   \hspace*{-.6cm}
\includegraphics[scale=0.35]{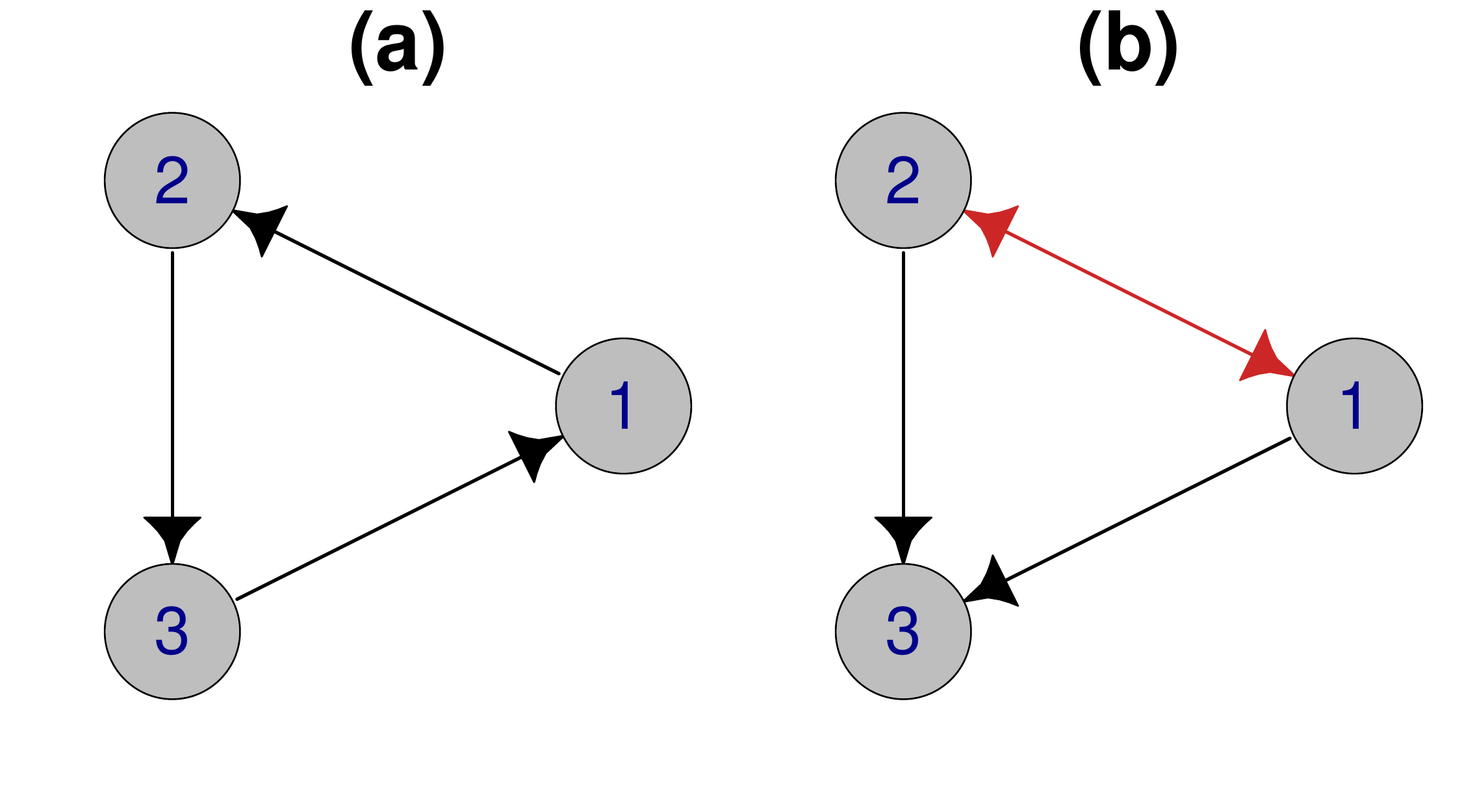}
\caption{Two simple mixed graphs.}
  \label{fig:two_graph}
\end{figure}

\begin{definition}
  \label{def:sem}
  The \emph{linear Gaussian model} given by the mixed graph
  $G=(V,D,B)$ is the family of all multivariate normal distributions
  on $\mathbb{R}^V$ with covariance matrix in 
  \begin{equation*}
    \mathcal{M}_G 
    = \left\{ (I-\Lambda)^{-T}\Omega(I-\Lambda)^{-1} \right.
      :  \left. 
      \Lambda\in\mathbb{R}^D_\mathrm{reg}, \,
      \Omega\in\mathit{PD}(B)\right\}.\nonumber
  \end{equation*}
  The \emph{covariance parametrization} of the model is the map
  \begin{align}
    \phi_G :  \mathbb{R}^D_{\rm reg} \times \mathit{PD}(B) 
    &\mapsto \mathit{PD}, \\
    (\Lambda, \Omega) &\mapsto (I - \Lambda)^{-\trans} \Omega (I -
                        \Lambda)^{-1}.\nonumber
  \end{align}
\end{definition}

Note that the model leaves the mean vector of the normal distribution unrestricted.  Without loss of generality, we may assume it to be zero.

It is well-known that different mixed graphs may induce the same model $\mathcal{M}_G$, which we describe as follows. 
\begin{definition}
Two mixed graphs $G_1$ and $G_2$ are distributionally equivalent if $\mathcal{M}_{G_1}=\mathcal{M}_{G_2}$. 
\end{definition}

Distributional equivalence is a stronger requirement than Markov equivalence, which only requires the same conditional independence relations. For acyclic simple mixed graphs, distributional equivalence is implied by having the same skeleton and collider triples \cite{nowzohour:2017}.

\section{Expected Dimension}
\label{sec:dimension}

Knowing the dimension of the set of covariance matrices $\mathcal{M}_G$ of a linear causal model is crucial for statistical model selection. In particular, the dimension features in model selection scores/information criteria \cite{sBIC}. Because $\mathcal{M}_G$ is defined by the pair of parameter matrices $(\Lambda,\Omega)$, it is natural to expect that its dimension $\dim(\mathcal{M}_G)$ equals the number of unknown parameters, which is $|V|+|D|+|B|$. For general graphs, however, this may fail to be true in subtle ways \cite{foygel:draisma:drton:2012}.  Nevertheless, the following theorem shows that this complication does not arise when considering simple graphs.

\begin{theorem}\label{thm:dim}
If the graph $G$ is simple, then
\[
\dim(\mathcal{M}_G) =|V|+|D|+|B|.
\]
\end{theorem}

Let $J_G$ be the Jacobian of the covariance parametrization $\phi_G$.  As discussed, e.g., in \cite{MR1863967}, the dimension of $\mathcal{M}_G$ equals the maximal rank of $J_G$.   
Before proving Theorem~\ref{thm:dim} based on  this fact, we extend observations from \cite{foygel:draisma:drton:2012} that simplify studying the rank of $J_G$.

On the domain $\mathbb{R}^D_{\rm reg}\times PD$, define the map
\begin{equation}
\label{eq:g}
g:(\Lambda,\Sigma) \mapsto(I-\Lambda)^T\Sigma(I-\Lambda).
\end{equation}
A positive definite matrix $\Sigma$ is in $\mathcal{M}_G$ if and only if there exist $\Lambda\in\mathbb{R}^D_\mathrm{reg}$ and $\Omega\in\mathit{PD}(B)$ such that  $\Sigma=\phi_G(\Lambda,\Omega)$, which holds if and only if $g(\Lambda,\Sigma)=\Omega$.
Let 
$$N=\{\{i,j\}: i,j\in V,\; i\neq j,\; \{i,j\}\notin B\}.$$
Then $\Sigma\in\mathcal{M}_G$ if and only if there exists $\Lambda\in\mathbb{R}^D_\mathrm{reg}$ such that 
\begin{equation}
    g_{ij}(\Lambda,\Sigma)= \left[(I-\Lambda)^T\Sigma(I-\Lambda)\right]_{ij}=0
\end{equation}
for all $\{i,j\}\in N$.  Consider then the $N\times D$ Jacobian matrix $\mathbf{J}(\Lambda,\Sigma)$ whose entries are the partial derivatives
\begin{equation}
    \mathbf{J}(\Lambda,\Sigma)_{\{i,j\},(k,l)} =
\frac{\partial g_{ij}(\Lambda,\Sigma)}{\partial \lambda_{kl}}
\end{equation}
with $\{i,j\}\in N$ and $k\to l\in D$.  For $i\not=j$, $g_{ij}$ is multilinear in $\Lambda$ and
\begin{align}\label{lem2.1.1}
   \frac{\partial g_{ij}(\Lambda,\Sigma)}{\partial \lambda_{kl}}
    &=
    \begin{cases}
    -[(I-\Lambda)^T\Sigma]_{jk} &\text{if} \ l = i,\\
    -[(I-\Lambda)^T\Sigma]_{ik} &\text{if} \ l = j,\\
    0 & \text{if} \ l\notin \{i,j\}.
    \end{cases}
\end{align}

\begin{lemma}\label{lem2.1.2}
    For $\Lambda\in\mathbb{R}^D_\mathrm{reg}$, $\Omega\in\mathit{PD}(B)$, let $\Sigma=\phi_G(\Lambda,\Omega)$. Then the rank of the Jacobian $J_G(\Lambda,\Omega)$ is equal to
    \[
    \mathrm{rank}\left(\mathbf{J}\left(\Lambda,\Sigma\right)\right)+|B|+|V|. 
    \]
\end{lemma}

\begin{proof}
On $\mathbb{R}^D_{\rm reg}\times \mathit{PD}(B)$, define the map
$$
h:(\Lambda,\Omega) \mapsto(\Lambda,\phi_G(\Lambda,\Omega)).
$$
Composing with $g$ from \eqref{eq:g}, we have 
\begin{align}\label{comp}
(g\circ h)(\Lambda,\Omega)=\Omega.
\end{align}
Differentiating this equation with respect to the free entries (i.e., nonzero) in $\Lambda$ gives
\begin{equation}
    \frac{\partial}{\partial \Lambda}
    g(\Lambda,\Sigma)|_{\Sigma=\phi_G(\Lambda,\Omega)}
    +\frac{\partial}{\partial \Sigma}
    g(\Lambda,\Sigma)|_{\Sigma=\phi_G(\Lambda,\Omega)}\frac{\partial}{\partial\Lambda}\phi_G(\Lambda,\Omega)\;=\;0.
\end{equation}
Similarly, differentiating with respect to the free entries of $\Omega$ gives
\begin{equation}
    \frac{\partial}{\partial \Sigma}
    g(\Lambda,\Sigma)|_{\Sigma=\phi_G(\Lambda,\Omega)}\frac{\partial}{\partial\Omega}\phi_G(\Lambda,\Omega)\;=\;
    \begin{pmatrix}
0\\
I_{|B|+|V|}
\end{pmatrix}.
\end{equation}
Here, the rows are indexed by unordered pairs $\{i,j\}$, due to the symmetry of the matrices in \eqref{comp}.  In the partitioning of the rows, the pairs in $N$ are listed first.

By (\ref{lem2.1.1}), again ordering the pairs in $N$ first, we have
\begin{align*}
\frac{\partial}{\partial\Lambda} g(\Lambda,\Sigma) = 
\begin{pmatrix}
\mathbf{J}(\Lambda,\Sigma) \\
0
\end{pmatrix}.
\end{align*}
For the Jacobian  $J_G=(\frac{\partial\phi_G}{\partial\Lambda},\frac{\partial\phi_G}{\partial\Omega})$, we obtain that
\begin{equation}
\label{eq:blockJacobian}
\frac{\partial}{\partial\Sigma} g(\Lambda,\Sigma)|_{\Sigma=\phi_G(\Lambda,\Omega)} \cdot J_G(\Lambda,\Omega)
=
\begin{pmatrix}
-\mathbf{J}(\Lambda,\Sigma)|_{\Sigma=\phi(\Lambda,\Omega)} & 0\\
0 & I_{|B|+|V|}
\end{pmatrix},
\end{equation}
with rows and columns partitioned as $(N, B\cup V)$ and $(D, B\cup V)$, respectively. Restricting $g$ by fixing $\Lambda\in\mathbb{R}^D_\mathrm{reg}$ gives the bijection $\Sigma \mapsto (I-\Lambda)^{T}\Sigma(I-\Lambda)$.  Hence, the matrix $\frac{\partial g}{\partial\Sigma}$ is invertible. It follows that the rank of $J_G$ equals the rank of the partitioned matrix on the right-hand side of \eqref{eq:blockJacobian}, which in turn equals
$\mathrm{rank}\left(\mathbf{J}\left(\Lambda,\Sigma\right)\right)+|B|+|V|$,
as was our claim.
\end{proof}

We now consider the Jacobian $J_G(0,I)$ obtained by specializing $\Lambda=0$ and $\Omega=I$.

\begin{lemma}\label{lem2.1.3}
 The Jacobian $J_G(0,I)$ has full column rank $|V|+|B|+|D|$ if and only if the mixed graph $G$ is simple.
 \end{lemma}

\begin{proof}
  The choice $\Lambda=0$ and $\Omega=I$ yields covariance matrix $\Sigma=\phi_G(0,I)=I$.  
  By Lemma~\ref{lem2.1.2}, it thus suffices to show that the matrix $\mathbf{J}(0,I)$ defined in \eqref{lem2.1.1} has full column rank $|D|$ if and only if $G$ is simple.
  
  Suppose $G$ is simple. Let $k\to l\in D$ be any directed edge indexing a column of $\mathbf{J}(0,I)$.  With $\Lambda=0$ and $\Sigma=I$, we have $-(I-\Lambda)^T\Sigma=-I$ and the considered column contains precisely one nonzero entry, namely,
  \[
    \mathbf{J}(0,I)_{\{k.l\},(k,l)}
    = -1;
  \]
  note that $k\to l\in D$ implies $\{k,l\}\in N$ if $G$ is simple.  Arranging the row indices $\{k,l\}$ for $k\to l\in D$ first, it becomes evident that $\mathbf{J}(0,I)$ has rank $|D|$ as 
  \begin{align*}
\mathbf{J}(0,I)=
\begin{pmatrix}
-I_D\\
0
\end{pmatrix}.
\end{align*}
  
  Conversely, suppose that $G$ is not simple, with $k$ and $l$ being two nodes joined by at least two edges.  Without loss of generality, assume that one of these edges is $k\to l\in D$.  We distinguish two cases.  First, suppose $k\leftrightarrow l\in B$.  Then $\{k,l\}\notin N$.  Therefore, the column of $\mathbf{J}(0,I)$ that is indexed by $(k,l)$ is zero, which implies that the rank of $\mathbf{J}(0,I)$ is smaller than $|D|$.  Second, suppose $k\leftrightarrow l\notin B$ but $l\to k\in D$. Then the two columns of $\mathbf{J}(0,I)$ indexed by $(k,l)$ and $(l,k)$ are identical.  Again the rank of $\mathbf{J}(0,I)$ is smaller than $|D|$. 
\end{proof}

\begin{proof}[Proof of Theorem \ref{thm:dim}]
    If $G$ is simple, then Lemma~\ref{lem2.1.3} implies that the maximal rank of the Jacobian $J_G$ is 
     $|V|+|D|+|B|$, which equals the count of parameters.
\end{proof}

For simple acyclic mixed graphs, having the same skeleton is a necessary condition for distributional equivalence  \cite{nowzohour:2017}, but this condition is not necessary for cyclic graphs \cite{ghassami2019characterizing}.  However, by Theorem \ref{thm:dim}, two distributionally equivalent simple cyclic mixed graphs must at least have the same number of edges.

\section{Sufficient Conditions for Distributional Equivalence}
\label{sec:sufficient}

 In this section, we show that the sufficient condition for distributional equivalence from \cite{nowzohour:2017} admits an extension to our setting of possibly cyclic graphs. To this end, we define $\overline{\mathcal{M}_G}$ to be the closure of $\mathcal{M}_G$ (in Euclidean topology). Two mixed graphs $G_1$ and $G_2$ are then distributionally equivalent up to closure if $\overline{\mathcal{M}_{G_1}}=\overline{\mathcal{M}_{G_2}}$.

\begin{theorem} \label{thm:equivalence}
 Let $G_{1}$ and $G_{2}$ be two simple mixed graphs with same skeleton and collider triples. Then $G_1$ and $G_2$ are distributionally equivalent up to closure.
 \end{theorem}

Note that the likelihood functions of two models that are equal up to closure have the same supremum.
While our proof of Theorem \ref{thm:equivalence} (developed in \S\ref{subsec:useful}-\ref{subsec:constructing}) concludes equality up to closure, we do not know any examples where the models are not exactly equal.  The condition in Theorem \ref{thm:equivalence} is also far from being necessary, e.g., it does not include the well-known characterization of Markov equivalence for directed acyclic graphs (DAGs).  However, the theorem is useful to assert equivalence in our simulations (Section~\ref{subsec:sims}).  We are also not aware of better (tractable) conditions in the literature.  Indeed, distributional equivalence for cyclic mixed graphs is a subtle problem as the following example shows.

\begin{example}
  Let $G_1$ and $G_2$ be the two simple mixed graphs displayed in Figure~\ref{fig:two_graph}(a) and (b), respectively.  By Theorem \ref{thm:dim}, both $\mathcal{M}_{G_1}$ and $\mathcal{M}_{G_2}$ are full-dimensional (i.e., 6-dimensional) subsets of the cone of positive definite $3\times 3$  matrices.  Graph $G_2$ is acyclic, and $\mathcal{M}_{G_2}$ is easily seen to be equal to $\mathit{PD}$.  However, as observed in \cite{drton:mlthreshold:2019}, the set $\mathcal{M}_{G_1}$ is a strict subset of $\mathcal{M}_{G_2}=\mathit{PD}$.
\end{example}

\subsection{Useful Lemmas}
\label{subsec:useful}

Let $G_1=(V,D_1,B_1)$ and $G_2=(V,D_2,B_2)$ be two mixed graphs.  Let $(\Lambda_1,\Omega_1)\in\mathbb{R}^{D_1}_{\mathrm{reg}}\times \mathit{PD}(B_1)$ be parameters for $G_1$.  The essence of the proof of Theorem \ref{thm:equivalence} is a strategy to find parameters $(\Lambda_2,\Omega_2)\in\mathbb{R}^{D_2}_{\mathrm{reg}}\times \mathit{PD}(B_2)$ such that $\phi_{G_2}(\Lambda_2,\Omega_2) = \phi_{G_1}(\Lambda_1,\Omega_1)$.  The key steps of the construction are a reduction to correlation matrices and an edge-relabeling considered in the acyclic case by \cite{nowzohour:2017}.  However, the cyclic case brings about new subtleties in this approach.

Let $\mathcal{R}: \mathit{PD} \to \mathit{PD}$ be the standardization map that takes covariance matrices to correlation matrices via $\mathcal{R}(\Sigma)_{ij} = \frac{\Sigma_{ij}}{\sqrt{\Sigma_{ii}\Sigma_{jj}}}$.

\begin{lemma} \label{lem:standardized}
Let $G=(V,D,B)$ be simple and $\Sigma \in \mathit{PD}$. Then 
$\Sigma \in \mathcal{M}_G  \, \text{ if and only if } \, \mathcal{R}(\Sigma) \in \mathcal{M}_G$.
\end{lemma}
\begin{proof}
We show one direction as the converse can be verified similarly. If $\Sigma \in \mathcal{M}_G $ then 
$$\Sigma=\phi_G(\Lambda,\Omega)=(I-\Lambda)^{-T} \Omega (I-\Lambda)^{-1}. $$
for some matrices $\Lambda \in \mathbb{R}^{D}_{\text{reg}} \, , \Omega \in PD(B)$. Setting $\Delta$ diagonal with entries $\Delta_{ii} = \Sigma_{ii}^{-\frac{1}{2}}$, it holds that
\begin{equation*}
\mathcal{R}(\Sigma) = \Delta \Sigma  \Delta
   = (\Delta^{-1}-\Delta^{-1}\Lambda)^{-T} \Omega (\Delta^{-1}-\Delta^{-1}\Lambda)^{-1}
     =  \phi_G(\tilde{\Lambda},\tilde{\Omega})
\end{equation*}
with $\tilde{\Lambda} = \Delta^{-1} \Lambda \Delta \in \mathbb{R}^{D}_{\text{reg}} \, , \,
\tilde{\Omega}=\Delta \Omega \Delta \in PD(B) $.
\end{proof}

Throughout the rest of this section, let $G_1=(V,D_1,B_1)$ and $G_2=(V,D_2,B_2)$ be two mixed graphs.  If the graphs have the same skeleton, then there is a natural way to copy the edge labels from one graph to the other. To describe the procedure, we decompose an error covariance matrix, $\Omega$,  into its diagonal and off-diagonal parts, denoted $\Omega^{d}$ and $\Omega^{od}$, respectively.  So, $\Omega = \Omega^{d} + \Omega^{od}$.

\begin{definition}
\label{eq:edgelabel}
Let $G_1$ and $G_2$ be simple mixed graphs with the same skeleton.  Given a choice $(\Lambda_{1},\Omega_{1})\in\mathbb{R}^{D_1}_{\mathrm{reg}}\times \mathit{PD}(B_1)$, the \emph{induced edge labeling on $G_2$} is the pair of matrices $(\Lambda_{2},\Omega_{2}^{od})$ obtained as
\begin{align*}
(\Lambda_{2})_{ij} \, =\begin{cases}
(\Lambda_{1})_{ij} & \text{if } i \rightarrow j \in G_{1}, \, i \rightarrow j \in G_{2},\\
(\Lambda_{1})_{ji} & \text{if } i \leftarrow j \in G_{1}, \, i \rightarrow j \in G_{2},\\
(\Omega_{1})_{ij} & \text{if } i \leftrightarrow j \in G_{1}, \, i \rightarrow j \in G_{2},\\
0 & \text{if } i \rightarrow j \notin G_{2},
\end{cases}
\end{align*}
\vspace{-0.3cm}
\begin{align*}
(\Omega_{2}^{od})_{ij}=\begin{cases}
(\Lambda_{1})_{ij} & \text{if } i \rightarrow j \in G_{1}, i \leftrightarrow j \in G_{2},\\
(\Lambda_{1})_{ji} & \text{if } i \leftarrow j \in G_{1}, i \leftrightarrow j \in G_{2},\\
(\Omega_{1})_{ij} & \text{if } i \leftrightarrow  j \in G_{1}, i \leftrightarrow j \in G_{2},\\
0 & \text{if } i \leftrightarrow j \notin G_{2} \ \text{or} \ i=j.
\end{cases}
\end{align*}
\end{definition}

For the construction from Definition~\ref{eq:edgelabel}, it holds that $\Lambda_2\in\mathbb{R}^{D_2}_{\mathrm{reg}}$.  Moreover, $\Omega_{2}^{od}$ can be turned into a matrix in $\mathit{PD}(B_2)$ by addition of a diagonal matrix.

\begin{lemma}
\label{lem:determinante}
Let $G_{1}$ and $G_{2}$ be simple mixed graphs with same skeleton and collider triples, and let  $(\Lambda_{i},\Omega_{i}) \in \mathbb{R}^{D_i} \times \mathit{PD}(B_i)$ for $i=1,2$.  If  $(\Lambda_{2},\Omega_{2}^{od})$ equals the edge labeling induced by $(\Lambda_{1},\Omega_{1})$ then
$$\det(I-\Lambda_{1})=\det(I-\Lambda_{2}).$$ 
In particular, if $\Lambda_1 \in \mathbb{R}^{D_1}_{\text{reg}}$ then $\Lambda_2 \in \mathbb{R}^{D_2}_{\text{reg}}$.
\end{lemma}
\begin{proof}
The determinants depend on the values of cycle products \cite[Lemma 1]{BCD}.  Let $S_V$ be the group of permutations of the nodes in $V$.  For $\sigma \in S_V$, let $V(\sigma)$ be the set of nodes contained in a non-trivial cycle of $\sigma$.  Then
\begin{equation}\label{eq:detcycle}
\det(I-\Lambda)=\sum_{\sigma \in S_{V}(G)}(-1)^{\mathrm{sgn}(\sigma)} \prod_{i \in V(\sigma)} \Lambda_{\sigma(i),i}
\end{equation}
where $S_{V}(G)$ is the subset of permutations such that $i=\sigma(i)$ or $i\rightarrow\sigma(i)\in D$ for all $i\in V$. We remark that even though the lemma in \cite{BCD} is stated for $\Lambda \in \mathbb{R}^{D}_{\text{reg}}$, the proof relies on Laplace expansion of the determinant which holds even if $I- \Lambda$ is not invertible.

Now, since collider triples are preserved, an edge that is part of a directed cycle of $G_1$ cannot be bidirected in $G_2$. Furthermore, if $G_1$ contains a cycle which has a directed edge that is reversed in $G_2$, then the cycle must be chordless in $G_1$ (that is, every node in the cycle can have only one child in the cycle) and must be fully reversed in $G_2$. Since the labels agree, the cycle products in \eqref{eq:detcycle} remain unchanged and therefore $\det(I-\Lambda_{1})=\det(I-\Lambda_{2})$.
\end{proof}

\subsection{Constructing Covariance Matrices}
\label{subsec:constructing}

The key to completing the proof of Theorem \ref{thm:equivalence} is to show that a correlation matrix obtained from a generic choice of parameters $(\Lambda_1,\Omega_1)\in\mathbb{R}^{D_1}_{\mathrm{reg}}\times \mathit{PD}(B_1)$ also belongs to $\mathcal{M}_{G_2}$. 
Let $\circ$ denote the Hadamard (entrywise) product of matrices, and define $\mathcal{H} : \mathbb{R}^{D}_{\text{reg}} \to \mathbb{R}^{V \times V}$ by
$$\mathcal{H}(\Lambda):=(I-\Lambda)^{-T} \circ (I-\Lambda)^{-T}.$$ 
We denote the spectral radius of a matrix $\Lambda$ by $\rho(\Lambda)$.

\begin{lemma}\label{lem:equal}
Let $G_1,G_2$ be simple mixed graphs with same skeleton and collider triples. Let $(\Lambda_1,\Omega_1) \in \mathbb{R}^{D_1}_{\text{reg}} \times PD(B_1)$ such that $\Sigma= \phi_{G_1}(\Lambda_1,\Omega_1) \in \mathcal{M}_{G_1}$ is a correlation matrix and consider the induced edge labeling $(\Lambda_2,\Omega_2^{od})$. If 
\begin{itemize}
    \item[(i)] $\rho(\Lambda_j)<1$ for $j=1,2$, and
    \item[(ii)] $\det(\mathcal{H}(\Lambda_2)) \neq 0$,
\end{itemize}
then there exists a unique diagonal matrix $\Omega_2^{d}$ such that with $\Omega_2=\Omega_2^d+\Omega_2^{od}$ it holds that $(\Lambda_2,\Omega_2) \in \mathbb{R}^{D_2}_{\text{reg}} \times PD(B_2) $ and $\Sigma = \phi_{G_2}(\Lambda_2,\Omega_2) \in \mathcal{M}_{G_2}$.
\end{lemma}

\begin{proof}
By Lemma \ref{lem:determinante}, we have indeed that $\Lambda_2 \in  \mathbb{R}^{D_2}_{\text{reg}}$. We need to construct $\Omega_2^{d}$ such that
\begin{equation*}
\phi_{G_2}(\Lambda_{2},\Omega_{2})
=(I-\Lambda_{2})^{-T} \Omega_{2}^{d} (I-\Lambda_{2})^{-1}
+ (I-\Lambda_{2})^{-T} \Omega_{2}^{od} (I-\Lambda_{2})^{-1} = \Sigma.
\end{equation*}
Since $\Sigma_{ii}=1$, this requires for all $i\in V$ that
\begin{equation*}
((I-\Lambda_{2})^{-T} \Omega_{2}^{d} (I-\Lambda_{2})^{-1})_{ii}
= 1-((I-\Lambda_{2})^{-T} \Omega_{2}^{od} (I-\Lambda_{2})^{-1})_{ii}.
\end{equation*}
Solving for the diagonal of $\Omega_2^d$ is equivalent (see \cite[Lemma 5.1.3]{horn_johnson_1991} ) to the linear system $Ax=b$ where
\begin{align*}
A=\mathcal{H}(\Lambda_2) = (I-\Lambda_{2})^{-T} \circ (I-\Lambda_{2})^{-T} 
\end{align*}
and the coordinates of the vector $b$ are $$b_i = 1-((I-\Lambda_{2})^{-T} \Omega_{2}^{od} (I-\Lambda_{2})^{-1})_{ii}.$$
By hypothesis, $\det(\mathcal{H}(\Lambda_2)) \neq 0$ and the system has a unique solution. 
It thus remains to show that $\phi_{G_2}(\Lambda_{2},\Omega_{2})$ also matches $\Sigma$ in all off-diagonal entries.

In general, if $\phi(\Lambda,\Omega)$ is a correlation matrix over a mixed graph $G$ and  $\rho(\Lambda)<1$, by \cite[Theorem 4]{nowzohour:2017},  the entries for $i \neq j$ are given by
\begin{equation}\label{eq:thmoff}
\phi_G(\Lambda,\Omega)_{ij}=\sum_{\tau \in S^{ij}}\prod_{s \rightarrow t \in \tau} \Lambda_{ts} \prod_{s \leftrightarrow t \in \tau} \Omega_{st},
\end{equation}  
where $S^{ij}_G$ is the set of simple treks from $i$ to $j$.  By assumption, $\rho(\Lambda_1),\rho(\Lambda_2)<1$, and we may apply the representation in \eqref{eq:thmoff} to $G_1$ and $G_2$.  In general,  $S_{G_1}^{ij}\not=S_{G_2}^{ij}$.  However, the fact that the graphs have the same skeleton and share collider triples implies that when replacing $(\Lambda,\Omega)$ by $(\Lambda_j, \Omega_j)$, $j=1,2$, in \eqref{eq:thmoff}, the induced edge labeling guarantees that the right hand sides of the expression are equal.  Hence,
$$\phi_{G_1}(\Lambda_1,\Omega_1) = \Sigma = \phi_{G_2}(\Lambda_2,\Omega_2)$$
as was the claim.
\end{proof}

With these preparations in place, we may complete the proof of the main result of this section.

\begin{proof}[Proof of Theorem \ref{thm:equivalence}]
      First, observe that the covariance parametrization $$\phi_{G_1}: \mathbb{R}^{D_1} \times PD(B_1) \to \mathcal{M}_{G_1} \subseteq \mathit{PD}$$
      is a rational map.  Next, consider the algebraic map $$\psi:  \mathcal{U} \subset \mathbb{R}^{D_1} \times PD(B_1) \to \mathcal{M}_{G_2} \subseteq \mathit{PD}$$  
    defined as follows. First, apply the standardization map on the parameter $(\Lambda_1, \Omega_1)$ to obtain $(\tilde{\Lambda}_1, \tilde{\Omega}_1,\Delta)$, as in the proof of Lemma \ref{lem:standardized}.  We denote this map by $\tilde{\mathcal{R}}$. As $(\tilde{\Lambda}_1, \tilde{\Omega}_1)$ define a correlation matrix we may obtain $(\tilde{\Lambda}_2, \tilde{\Omega}_2)$ from the procedure in Lemma \ref{lem:equal}, for representation of the same correlation matrix.  Finally, destandardize $(\tilde{\Lambda}_2, \tilde{\Omega}_2)$ with the matrix $\Delta$ from the standardization map, and apply $\phi_{G_2}$.
    
    Note that the map $\psi$ is well-defined for input that satisfies the two conditions in Lemma \ref{lem:equal}.  This domain includes an open subset $\mathcal{U} \subset \mathbb{R}^{D_1} \times PD(B_1)$.  This subset is nonempty because $(0,I)\in \mathcal{U}$.  The final application of $\phi_{G_2}$ to $(\Lambda_2, \Omega_2)$ gives a matrix in $\mathcal{M}_{G_2}$, which by construction and Lemma \ref{lem:equal} coincides with $\phi_{G_1}(\Lambda_1, \Omega_1)$. The diagram in Figure \ref{fig:diag} illustrates the situation. 
    
\begin{figure}[t]
\centerline{
\begin{xy}
\xymatrix{
    (\Lambda_{1},\Omega_{1}) \ar[r]^{\tilde{\mathcal{R}}} \ar[ddrr]_{\phi_{G_{1}}}  & (\tilde{\Lambda}_{1},\tilde{\Omega}_{1},\Delta) \ar[r]^H & (\tilde{\Lambda}_{2},\tilde{\Omega}_{2},\Delta) \ar[d]^{\tilde{\mathcal{R}}^{-1}}\\
  & & (\Lambda_{2},\Omega_{2}) \ar[d]^{\phi_{G_{2}}}\\
  & & \Sigma_{1} = \Sigma_{2}
  }
 \end{xy}} 
 \caption{\label{fig:diag} Commutative diagram illustrating two ways of obtaining a matrix $\Sigma \in \mathcal{M}_{G_1}\cap \mathcal{M}_{G_2}$. One is the parametrization $\phi_{G_{1}}$, while the other is a composition of maps (including the map $H$ defined via Lemma \ref{lem:equal}), that we denote $\psi$.}
\end{figure}

 The map $\psi$ is a composition of a rational map with algebraic maps that involve radicals (i.e., square roots in the standardization $\mathcal{R}$). Since $\psi$ coincides with the rational map $\phi_{G_{1}}$ on the open set $\mathcal{U}$, they must be equal outside of an algebraic hypersurface (i.e., the zero set of a multivariate polynomial). This exceptional set has Lebesgue measure zero (see, e.g., the lemma in \cite{Okamoto}).  
 Covariance matrices in $\mathcal{M}_{G_1}$ that are given by parameters $(\Lambda_1,\Omega_1)$ outside the exceptional set are also in $\mathcal{M}_{G_2}$.  We may conclude that $\mathcal{M}_{G_1} \subseteq\overline{\mathcal{M}_{G_2}}$ because the elements of the exceptional set are limits of sequences off the exceptional set.  By symmetry,  $\overline{\mathcal{M}_{G_1}}=\overline{\mathcal{M}_{G_2}}$ as claimed.
\end{proof}

\section{Greedy Search}
\label{sec:modelsel}

Based on our result on the dimension of models given by simple mixed graphs, we may form scores that trade off model dimension and model fit.  For model selection, we may then maximize such a score over the considered set of graphs.  Given the large number of possible graphs, we follow prior work and consider a greedy search that starts from some initial graph and iteratively selects the highest-scoring graph from a local neighborhood of graphs. The procedure stops when no higher score can be found in the local neighborhood or a fixed maximum number of iterations is reached.  To mitigate getting trapped in local optima, the search is running from different (random) starting points.  In the present context, we take the local neighborhood of a graph $G$ to be the union of all simple mixed graphs that can be obtained from $G$ by adding one edge, by removing one edge, or by reversing one directed edge; compare \cite{nowzohour:2017}.

Let $\mathbf{X}\in\mathbb{R}^{n\times p}$ be a data matrix, assumed to hold in its rows the realizations of $n$ i.i.d.~and centered Gaussian random vectors.  Let $S=\mathbf{X}^T\mathbf{X}/n$ be the sample covariance matrix. The Gaussian log-likelihood function is
\begin{align*}
    \ell(\Sigma; S) = -\frac{n}{2} \left[\log\det(2\pi\Sigma)+ 
    \tr({\Sigma}^{-1} S)\right].
\end{align*}
Our proposed score for a mixed graph $G=(V,D,B)$ then takes the form
\begin{align}\label{score}
s(G)=
&\frac{1}{n}\left(\max_{\Sigma\in\mathcal{M}_G}\ell(\Sigma; S) - \text{penalty}(p,k,n)\right),
\end{align}
where $p=|V|$ and $k=|D|+|B|$ is the number of edges.  To compute the maximum log-likelihood in \eqref{score} we apply the block coordinate-descent algorithm from \cite{BCD}.  The standard Bayesian information criterion (BIC) uses $\text{penalty}(p,k,n)=\tfrac{1}{2}(p+k)\log(n)$; here $p+k$ is the model dimension. The authors of \cite{nowzohour:2017} double this penalty when searching over acyclic simple mixed graphs as it improved performance in experiments.  An increased penalty is supported by related work on selecting sparse graphical models \cite{Foygel:2010}. 

There the increased penalty is induced from a prior distribution over graphs under which the number of edges is uniformly distributed.  Adopting these ideas in our case we propose to define the score $s(G)$ from \eqref{score} with 
\begin{equation}
\label{eq:penalty}
\text{penalty}(p,k,n)=\tfrac{1}{2}(p+k)\log(n) + \log(p^{2k} 3^{k}).
\end{equation}
The last term reflects that there are $\binom{p(p+1)/2}{k} 3^k\sim p^{2k} 3^k$ simple mixed graphs with $k$ edges.  This penalty is formulated with a view towards sparser graphs as encountered in the application we consider in Section~\ref{subsec:protein}.  
We will also explore its use in simulated non-sparse problems of small scale (Section \ref{subsec:sims}). Note also that the proposed penalty ignores the issue of distributional equivalence, i.e., different graphs inducing the same model $\mathcal{M}_G$.  Unfortunately this equivalence issue is still poorly understood.

\section{Numerical Experiments}
\label{sec:numerical}

In this section we present numerical experiments, in which we apply the proposed greedy search to simulated data and well-known protein expression data \cite{sachs:2005}.

\subsection{Simulation Studies}
\label{subsec:sims}

We consider graphs with $p\in\{5,6\}$ nodes.  In each case generate $100$ simple mixed graphs uniformly at random using an MCMC algorithm; in analogy to \cite{nowzohour:2017}.  The parameters for the graphs' edges are sampled uniformly from $[-0.9, -0.5]\cup[0.5, 0.9]$. The diagonal entries in $\Omega$ are set by adding independent $\chi^2_1$ draws to the absolute row sums to ensure positive definiteness by diagonal dominance.  For each graph, we generate three Gaussian data sets of size $n \in \{10^2, 10^3, 10^4\}$.  

\begin{table}[htp]
\centering
\vspace{0.25cm}
\begin{tabular}{ccccccc}

\toprule

\textbf{BIC} & \textbf{n}             & \textbf{Start} & \textbf{Dim} & \textbf{Skel} & \textbf{Skel \& Coll}        & \textbf{SHD*}
                                      \\ \hline
\multirow{6}{*}{1}     & \multirow{2}{*}{$10^2$}   & \textbf{R}                      & \textbf{0.39}        & \textbf{0.13}         & \textbf{0.07} & \textbf{3.79}                             \\ 
                       &                        & TG                     & 0.8                  & 0.8                   & 0.25          & 1.15                                      \\ \cline{2-7} 
                       & \multirow{2}{*}{$10^3$}  & \textbf{R}                      & \textbf{0.63}        & \textbf{0.43}         & \textbf{0.26} & \textbf{2.44}                             \\ 
                       &                        & TG                     & 0.88                 & 0.88                  & 0.53          & 0.63                                      \\ \cline{2-7} 
                       & \multirow{2}{*}{$10^4$} & \textbf{R}                      & \textbf{0.76}        & \textbf{0.59}         & \textbf{0.45} & \textbf{2.29}                             \\ 
                       &                        & TG                     & 0.92                 & 0.92                  & 0.74          & 0.34                                      \\ \hline
\multirow{6}{*}{2}     & \multirow{2}{*}{$10^2$}   & \textbf{R}                      & \textbf{0.24}        & \textbf{0.14}         & \textbf{0.1}  & \textbf{3.55}                             \\ 
                       &                        & TG                     & 0.92                 & 0.92                  & 0.36          & 1.03                                      \\ \cline{2-7} 
                       & \multirow{2}{*}{$10^3$}  & \textbf{R}                      & \textbf{0.48}        & \textbf{0.34}         & \textbf{0.21} & \textbf{2.78}                             \\  
                       &                        & TG                     & 0.9                  & 0.9                   & 0.52          & 0.65                                      \\ \cline{2-7} 
                       & \multirow{2}{*}{$10^4$} & \textbf{R}                      & \textbf{0.71}        & \textbf{0.61}         & \textbf{0.38} & \textbf{2.02}                             \\ 
                       &                        & TG                     & 0.93                 & 0.93                  & 0.71          & 0.42                                       \\
                       \bottomrule
\end{tabular}
\caption[Statistics on p=5]{\label{p5} Proportion of estimated graphs that share the dimension (Dim), skeleton (Skel) and both skeleton and set of collider triples (Skel \& Coll) with the true graph, and minimal structural hamming distance (SHD*) averaged over simulations.  Estimates use BIC with standard (1) and increased penalty (2), and search initialized at random (R) or at the true graph (TG).}
\end{table}

For each data set, we run the greedy search starting from $(i)$ $300$ randomly selected graphs but also from $(ii)$ the true graph. 
For every single restart, we set the maximum number of iterations of greedy search to be $10^4$.  
In Table \ref{p5} we report on  how often the greedy search gives a model of correct dimension, correct graph skeleton, and both correct collider triples and skeleton when $p=5$.  For each initialization scheme, we consider the BIC with both standard penalty and penalty as in (\ref{eq:penalty}). Furthermore, in Table \ref{frequency5}, the frequency distribution of the difference between the dimension for the true graph and for the estimated graph  is reported for the case $p=5$ and BIC with standard penalty.

In Table \ref{p5}, we also report a structural Hamming distance (SHD); counting edge additions, deletions and reversals needed to move from one mixed graph to another.  The distance we give uses our theorem on skeleton and collider triples to bound the true minimal SHD between a graph representing the selected model and one representing the true model.  In other words, we minimize the SHD over pairs of graphs $(\bar G_1,\bar G_2)$, where $\bar G_1$ has the same skeleton and collider triples as the true graph and $\bar G_2$ has the same skeleton and collider triples as the estimated graph.  While this upper bound SHD* needs not be tight, it is on average only half as large as a naive SHD computed for estimated and true graph directly.  

By Theorem \ref{thm:dim}, the frequency of having the same dimension gives an upper bound for the frequency of getting equivalent models. Additionally, Theorem \ref{thm:equivalence} indicates that the frequency of having both the same set of collider triples and the same skeleton is a lower bound for the frequency of getting equivalent models. Hence, according to our experiment with $100$ simulations, when starting from the true graph, the estimated graph is distributionally equivalent to the true graph between $74\%$ and $92\%$ of times when $p=5$ and $n=10^4$ (BIC with standard penalty, see Table \ref{p5}) and between $61\%$ and $80\%$ when $p=6$ (results not given in Table). On the other hand, if the greedy search algorithm is started from a random graph, the estimated graph belongs to the equivalence class of the true graph between $45\%$ and $76\%$ of times when $p=5$ and between $21\%$ and $60\%$ when $p=6$.  The standard penalty seems to slightly outperform the increased penalty when $p \in \{5,6\}$, although this is not so evident from the upper bound on the true minimal structural Hamming distance. 
Additional partial experiments we carried out for $p=10$ suggest that for a larger number of nodes, smaller Hamming distances result from the increased penalty. 

\begin{table}[htp]
\centering
\begin{tabular}{cccccccc}
\toprule
\textbf{n} &
 \textbf{Start} &
  \multicolumn{5}{c}{\textbf{Dim(EST) - Dim(TG)}} \\ \cmidrule(lr){3-8}   
                     &        & \textbf{-3} & \textbf{-2} & \textbf{-1} & \textbf{0}  & \textbf{1} & \textbf{2}  \\ \midrule
\multirow{2}{*}{$10^2$} & \textbf{R} & \textbf{5}  & \textbf{14}  & \textbf{35} & \textbf{39} & \textbf{7} & \textbf{0} \\ 

                     & TG & 0  &0  &0  & 80 &  20 & 0 \\ \midrule
                     
\multirow{2}{*}{$10^3$} & \textbf{R} & \textbf{1}  & \textbf{3} & \textbf{25} &  \textbf{63}  & \textbf{8} & \textbf{0} \\ 
                     & TG & 0 & 0 & 0 & 88 & 12& 0 \\ \midrule
                     
 \multirow{2}{*}{$10^4$} & \textbf{R} & \textbf{0}  & \textbf{2} & \textbf{14} &  \textbf{76}  & \textbf{8} & \textbf{0}  \\ 
                     & TG & 0 & 0 & 0 & 97 & 7 & 1 \\ \bottomrule
\end{tabular}
\caption{\label{frequency5}Absolute frequency distribution of the difference between the dimension of the estimated graph (EST) and the dimension for the true graph (TG) in $100$ simulations for $p=5$ and BIC with standard penalty.}
\end{table}

\subsection{Protein Expression Data}
\label{subsec:protein}

For further illustration, we consider a frequently studied collection of data sets on expression of $p=11$ proteins in human T-cells \cite{sachs:2005}.  The collection comprises 14 data sets, each obtained under different experimental conditions. The sample size of these data sets ranges from $707$ to $927$. The focus on $11$ proteins is due to limitations in the experimental technology. Figure $2$ in \cite{sachs:2005} suggests presence of further relevant but unobserved proteins and leaves open the possibility of a feedback cycle.  
 
To accommodate departures from our linear Gaussian models, at least at the level of marginal distributions, we consider a Gaussian copula version of the models \cite{liu:2012,harris2013pc}. In other words, each observed variable is assumed to be a deterministic and isotonic function of a Gaussian latent variable. As shown in \cite{liu:2012}, consistent estimation in the Gaussian copula models is achieved by replacing the sample covariance matrix in the Gaussian likelihood function (and the model selection score) by a bias-corrected Kendall's tau correlation matrix.  The entries of this matrix are $\sin(\frac{\pi}{2}\hat{\tau}_{ij})$, where $\hat{\tau}_{ij}$ is Kendall's $\tau$ for the pair of variables  $(X_{i},X_{j})$.  With this substitution, we actually project a $\frac{p(p+1)}{2}$-dimensional covariance matrix to the $\frac{p(p-1)}{2}$-dimensional space of correlation matrices.  However, our sufficient condition for equivalence remains unchanged, and the result on expected dimension still holds as long as the graph has no more than $\frac{p(p-3)}{2}$ edges.

\begin{table}[t]
\centering

\begin{tabular}{cccc}
\toprule
\textbf{Type of edges}  & \textbf{Min} & \textbf{Median}& \textbf{Max}\\
\midrule
All &10  &13 &13\\
Directed &2  &8 &11\\
Bidirected & 2 &4.5 &8\\
\bottomrule
\end{tabular}

\caption[Statistics on the Number of Edges]{\label{tab2}  Summary statistics on the number of edges in the estimated graphs for the $14$ protein expression datasets.}

\end{table}

For each dataset, the greedy search based on the bias-corrected Kendall's tau matrix was repeated 
$300$ 
times, each time starting from a random graph and using BIC with increased penalty. 
The highest scoring graph for each data set was then determined.  
As reported in Table \ref{tab2}, the total number of edges in the $14$ estimated graph ranges from $10$ to $13$, with the median being $13$.  The table also gives these statistics for the count of directed and bidirected edges. The proportion of directed edges in each graph ranges from a minimum of $0.2$ to a maximum of $0.85$. A total of $4$ of the $14$ graphs contain a directed cycle: datasets $3$, $6$ and $7$ each yield a graph with one $3$-cycles, and dataset $4$ leads to one $4$-cycle.

We display two of the selected graphs in Figure \ref{fig:min_max}, one with  minimum (dataset $1$) and one with maximum (dataset $6$) number of edges, the latter also displaying a 3-cycle.
Although further work is needed to fully determine possible equivalences, there is no obvious reason (e.g., by Theorem \ref{thm:equivalence}) for a distributionally equivalent graph without a cycle to exist.  We conjecture that this is indeed not the case.
Considering all $14$ graph estimates together it is reassuring to observe that some structure is shared. Figure \ref{fig:skeleton} shows the (undirected) edges that appear in at least $9$ and $13$ of the skeletons of the estimated graphs.

 \begin{figure}[htp]
  \center
\includegraphics[scale = .65]{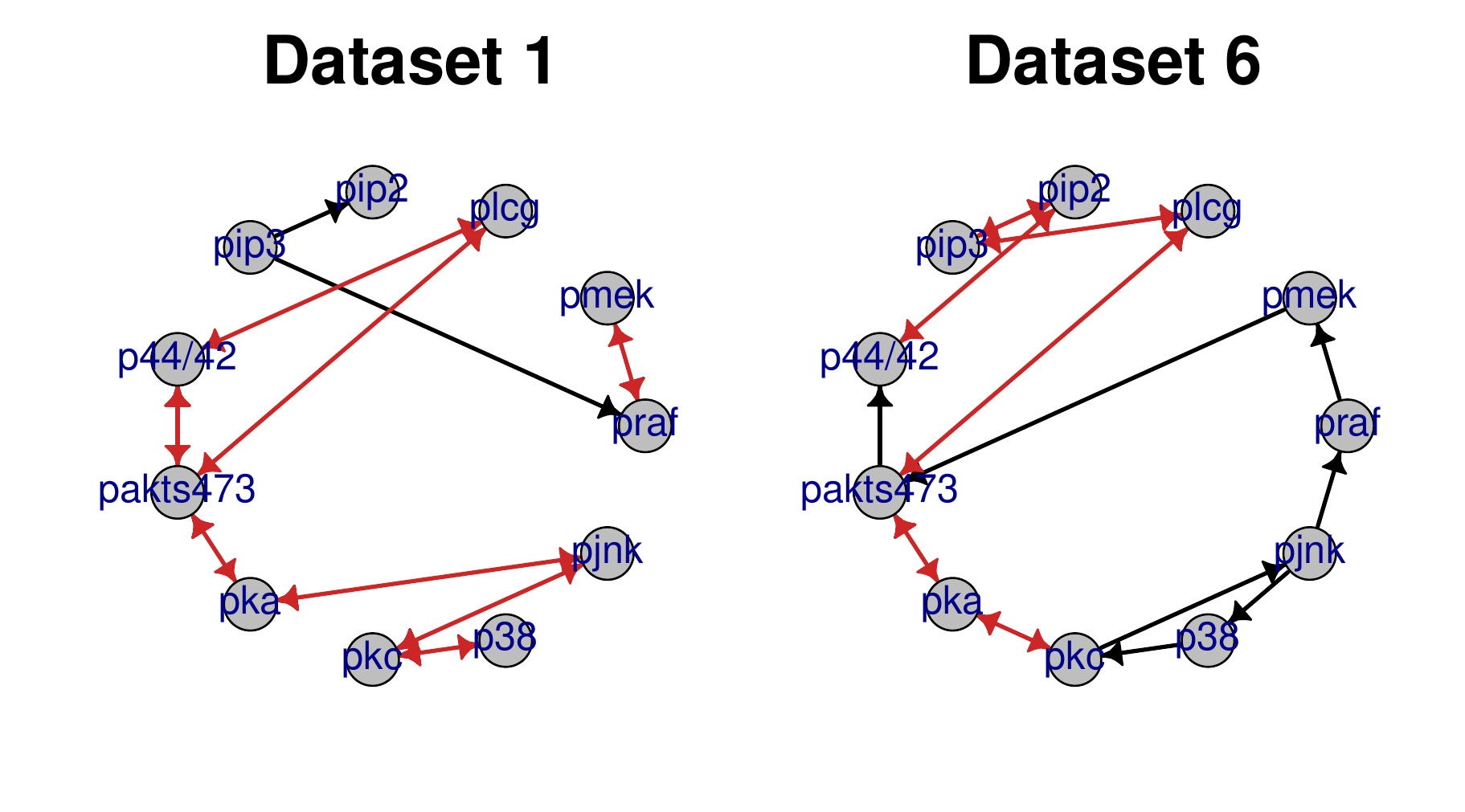}
\vspace*{-.3cm}
\caption{Estimated graphs corresponding to the case of minimum number of edges ($10$ edges, dataset $1$) and maximum number of edges ($13$ edges, dataset $6$).}
  \label{fig:min_max}
\end{figure}

Our selected graphs show good agreement with regulatory relationships described in \cite{sachs:2005}, e.g., the interplay PLCG-PIP2-PIP3 (in at least $12$ of the inferred graphs); the connection PKC-P38-PJNK (all $14$ graphs); the connection P44/42 (named ERK in \cite{sachs:2005}) and PKA-PAKTS473 (named AKT in \cite{sachs:2005}, all $14$ graphs). Moreover, three expected relationships that are well-reported from the field-related literature emerge in our work that were undetected in \cite{sachs:2005}. 
This is the case for the connections: PIP2 to PKC (dataset $14$), PLCG to PKC (dataset $14$) and PIP3 to PAKTS473 (dataset $6$, $7$ and $12$).

 \begin{figure}[htp]
  \center

\includegraphics[scale=.65]{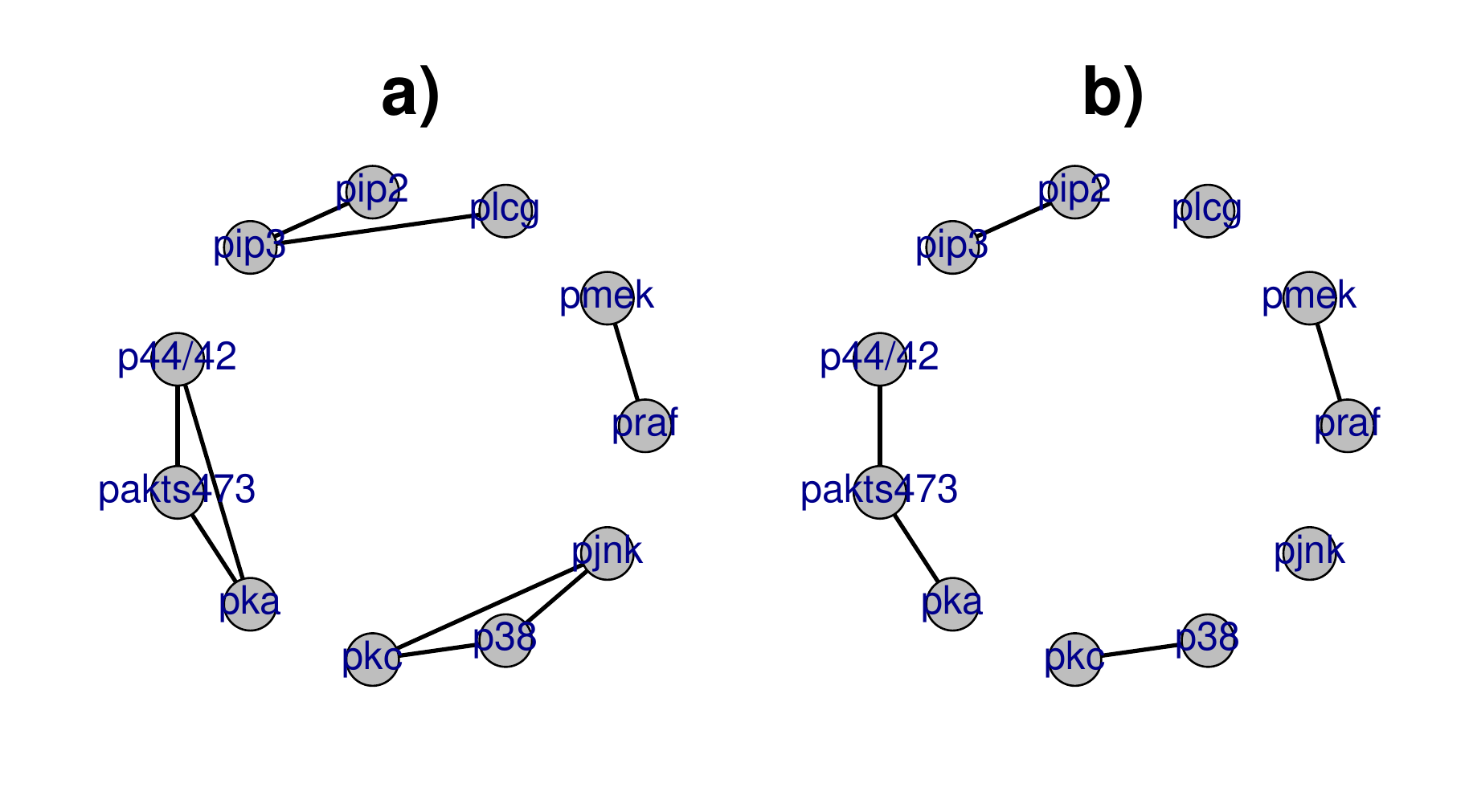}
\vspace*{-.3cm}
\caption{Edges appearing in at least $9$ (a) and 13 (b) skeletons from the estimated graphs.}
  \label{fig:skeleton}
 \end{figure}

Finally, in order to illustrate the behavior of the greedy search itself we focus again on datasets $1$ and $6$.  Figure \ref{fig:greedy} shows the respective search paths in terms of the score achieved at each iteration.  While local optima are possible, we observe that most search paths end with a score near the overall maximum.

 \begin{figure}[htp] 
  \center
\includegraphics[scale=0.9]{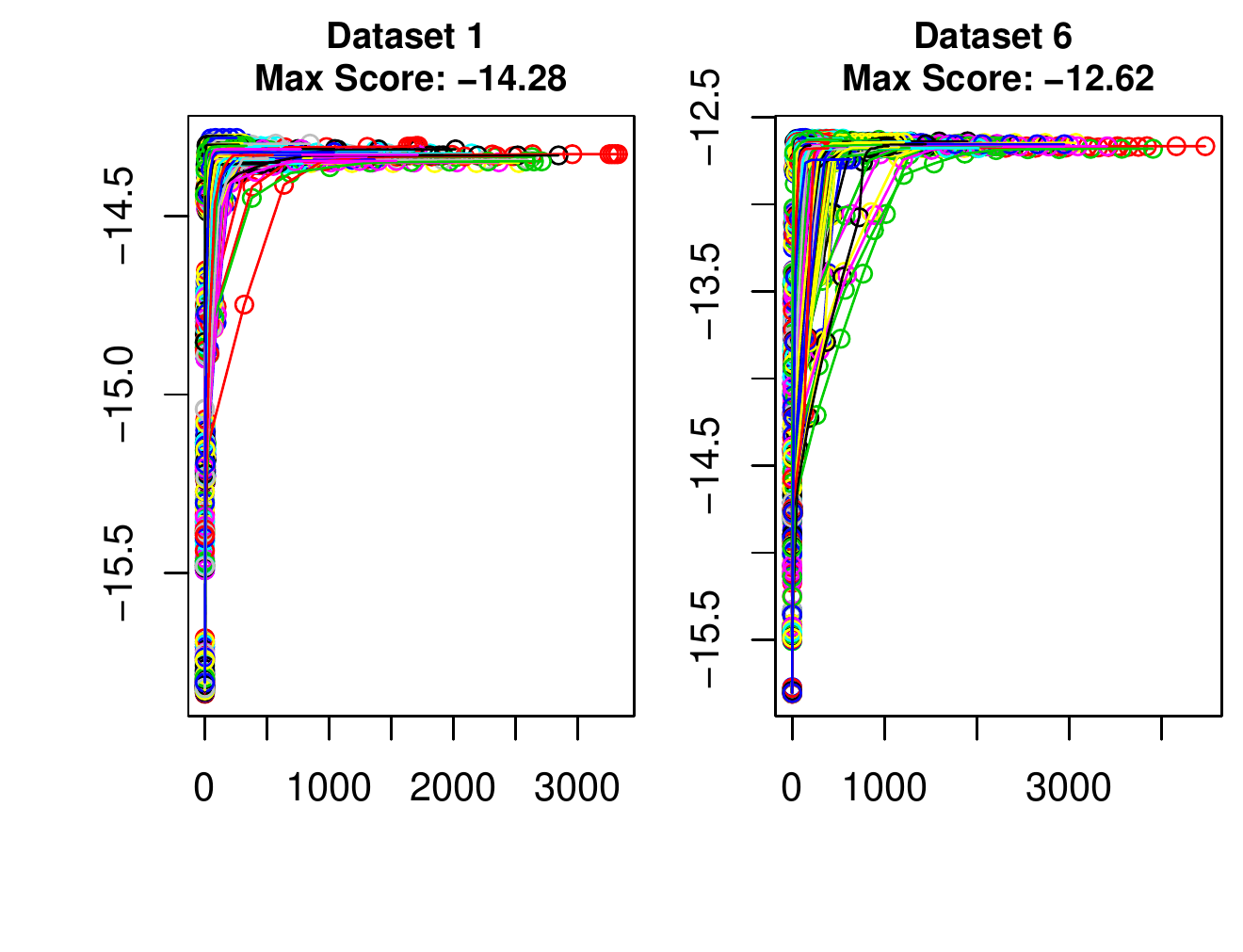}
\caption{Curves of scores versus the time (seconds) in 300 random restarts greedy search for dataset $1$ and dataset $6$. \label{fig:greedy}}
  \end{figure}

\section{Conclusion}
\label{sec:conclusion}
We considered structure learning for linear causal models with Gaussian errors that may exhibit feedback loops and correlation induced by latent variables.  In order to gain tractability in this difficult problem, we restricted our attention to simple mixed graphs.  Such graphs have the favorable property of always inducing a model whose dimension is as one expects from counting parameters.  This property allows one to form meaningful model selection scores.  While a search over simple mixed graphs remains challenging, computationally and statistically, our experiments suggest that useful information can be learned from greedy search methods.  This generalizes similar conclusions for acyclic simple graphs \cite{nowzohour:2017}.

We also showed that an existing sufficient condition for distributional equivalence admits a natural generalization from acyclic to cyclic simple mixed graphs.  However, the condition is very restrictive.  It would be important to find more broadly applicable conditions for distributional equivalence.

\subsubsection*{Acknowledgements}
Carlos Am\'endola was partially supported by the Deutsche Forschungsgemeinschaft (DFG) in the context of the Emmy Noether junior research group KR 4512/1-1. Federica Onori would like to thank Technical University of Munich for hosting her and University of Rome La Sapienza to make her research visit in Germany possible. We are grateful to anonymous reviewers for their constructive feedback on the paper.

\subsection*{References}
\renewcommand{\bibsection}{}
\bibliography{arxiv}
 
\bigskip 

\bigskip

\small {\bf Authors' addresses:}

\bigskip 

\noindent
\noindent Technische Universit\"at M\"unchen,
\hfill {\tt carlos.amendola@tum.de} \\
Chair of Mathematical Statistics, \hfill {\tt philipp.dettling@tum.de} \\
Department of Mathematics, \hfill {\tt mathias.drton@tum.de} \\
Germany \hfill {\tt jun1.wu@tum.de} \\

\noindent Universit\`a La Sapienza di Roma, 
\hfill {\tt onori.federica@gmail.com}\\
Faculty of 
Information Engineering, \\
Department of Statistical Sciences\\
Italy

\end{document}